\documentclass[reqno]{amsart}
\usepackage{hyperref}
\usepackage{amssymb}

\allowdisplaybreaks

\theoremstyle{plain}
\newtheorem{theorem}{Theorem}[section]

\newtheorem{lemma}[theorem]{Lemma}

\theoremstyle{definition}
\newtheorem{definition}[theorem]{Definition}

\theoremstyle{remark}
\newtheorem{remark}[theorem]{Remark}
\numberwithin{equation}{section}
\begin{document}
\title[Parabolic equation with nonlinear nonlocal boundary condition]
{Global existence and blow-up of solutions for semilinear heat
equation with nonlinear nonlocal boundary condition}

\author[A. Gladkov]{Alexander Gladkov}
\address{Alexander Gladkov \\ Department of Mechanics and Mathematics
\\ Belarusian State University \\ Nezavisimosti Avenue 4, 220030
Minsk, Belarus} \email{gladkoval@mail.ru}

\author[T. Kavitova]{Tatiana Kavitova}
\address{Tatiana Kavitova \\ Department of Mathematics \\ Vitebsk State University \\ Moskovskii pr. 33, 210038
Vitebsk, Belarus} \email{kavitovatv@tut.by}


\begin{abstract}
In this paper we consider a semilinear parabolic equation with
nonlinear and nonlocal boundary condition and nonnegative initial
datum. We prove some global existence results. Criteria on this
problem which determine whether the solutions blow up in finite
time for large or for all nontrivial initial data are also given.
\end{abstract}

\maketitle

\section{Introduction}\label{in}
In this paper we consider the initial boundary value problem for
the following semilinear parabolic equation
\begin{equation}\label{v:u}
    u_t=\Delta u+c(x,t)u^p,\;x\in\Omega,\;t>0,
\end{equation}
with nonlinear nonlocal boundary condition
\begin{equation}\label{v:g}
\frac{\partial u(x,t)}{\partial\nu}=\int_{\Omega}{k(x,y,t)u^l(y,t)}\,dy,\;x\in\partial\Omega,\;t>0,
\end{equation}
and initial datum
\begin{equation}\label{v:n}
    u(x,0)=u_{0}(x),\; x\in\Omega,
\end{equation}
where $p>0,\,l>0$, $\Omega$ is a bounded domain in $\mathbb{R}^n$
for $n\geq1$ with smooth boundary $\partial\Omega$, $\nu$ is unit
outward normal on $\partial\Omega.$

Throughout this paper we suppose that the functions
$c(x,t),\;k(x,y,t)$ and $u_0(x)$ satisfy the following conditions:
\begin{equation*}
c(x,t)\in
C^{\alpha}_{loc}(\overline{\Omega}\times[0,+\infty)),\;0<\alpha<1,\;c(x,t)\geq0;
\end{equation*}
\begin{equation*}
k(x, y, t)\in
C(\partial\Omega\times\overline{\Omega}\times[0,+\infty)),\;k(x,y,t)\geq0;
\end{equation*}
\begin{equation*}
u_0(x)\in C^1(\overline{\Omega}),\;u_0(x)\geq0 \textrm{ in }
\Omega,  \;\frac{\partial u_0(x)}{\partial\nu}=\int_{\Omega}{k(x,
y,0)u_0^l(y)}\,dy \textrm{ on } \partial\Omega.
\end{equation*}
Many authors have studied blow-up problem for parabolic equations
and systems with nonlocal boundary conditions (see, for
example, \cite{CY,CYZ,Deng_Dong,FZ,GG,GladkovGuedda,GladkovKim,Liu,Pao,WMX,YF,YX,Y,ZK}
and the references therein). In particular, the initial boundary
value problem for equation~(\ref{v:u}) with nonlinear nonlocal
boundary condition
\begin{equation*}
u(x,t)=\int_{\Omega}k(x,y,t)u^l(y,t)\,dy,\;x\in\partial\Omega,\;t>0,
\end{equation*}
was considered for $c(x,t)\leq0$ and $c(x,t)\geq0$
in~\cite{GladkovGuedda} and~\cite{GladkovKim} respectively.

Local existence theorem, comparison principle, the uniqueness and
nonuniqueness of solution for problem~(\ref{v:u})--(\ref{v:n})
have been considered in~\cite{GladkovKavitova}.

In this paper we obtain necessary and sufficient conditions for
the existence of global solutions as well as for a blow-up of
solutions in finite time for problem~(\ref{v:u})--(\ref{v:n}). Our
global existence and blow-up results depend on the behavior of the
functions $c(x,t)$ and $k(x,y,t)$ as $t \to \infty.$

This paper is organized as follows. In the next section we show
that all nonnegative solutions are global for $\max(p,l) \leq 1.$
In Section~\ref{mm} we prove blow-up of solutions for
large and for all nontrivial initial data as well as global
existence of solutions for small initial data.
Finally, in Section~\ref{blow} we establish that if $p\leq1$ and
$l>1$ blow-up can occur only on the boundary.

\section{Global existence}\label{gl}

Let $Q_T=\Omega\times(0,T),\;S_T=\partial\Omega\times(0,T)$,
$\Gamma_T=S_T\cup\overline\Omega\times\{0\}$, $T>0$.
\begin{definition}\label{v:sup}
We say that a nonnegative function $u(x,t)\in C^{2,1}(Q_T)\cap
C^{1,0}(Q_T\cup\Gamma_T)$ is a supersolution
of~(\ref{v:u})--(\ref{v:n}) in $Q_{T}$ if
        \begin{equation}\label{v:sup^u}
        u_{t}\geq\Delta u+c(x, t)u^{p},\;(x,t)\in Q_T,
        \end{equation}
        \begin{equation}\label{v:sup^g}
        \frac{\partial u(x,t)}{\partial\nu}\geq\int_{\Omega}{k(x, y, t)u^l(y, t) }\,dy,\;(x,t)\in S_T,
        \end{equation}
        \begin{equation}\label{v:sup^n}
            u(x,0)\geq u_{0}(x),\; x\in\Omega,
        \end{equation}
     and $u(x,t)\in C^{2,1}(Q_T)\cap C^{1,0}(Q_T\cup\Gamma_T)$ is a subsolution of~(\ref{v:u})--(\ref{v:n}) in $Q_{T}$ if $u\geq0$ and it satisfies~(\ref{v:sup^u})--(\ref{v:sup^n}) in the reverse order. We say that $u(x,t)$ is a solution of problem~(\ref{v:u})--(\ref{v:n}) in $Q_T$ if $u(x,t)$ is both a subsolution and a supersolution of~(\ref{v:u})--(\ref{v:n}) in $Q_{T}$.
\end{definition}

To prove the main results we use the positiveness of solution and
the comparison principle which have been proved
in~\cite{GladkovKavitova}.
\begin{theorem}\label{p:theorem:positive}
Suppose that $u_0\not\equiv0$ in $\Omega$ and $u(x,t)$ is a
solution of~(\ref{v:u})--(\ref{v:n}) in $Q_{T}$. Then $u(x,t)>0$
in $Q_T\cup S_T$.
\end{theorem}
\begin{theorem}\label{p:theorem:comp-prins}
Let $u(x,t)$ and $v(x,t)$ be a supersolution and a subsolution of
problem~(\ref{v:u})--(\ref{v:n}) in $Q_{T},$ respectively. Suppose
that $u(x,t)>0$ or $v(x,t)>0$ in $Q_T\cup\Gamma_T$ if
$\min(p,l)<1.$ Then $u(x,t)\geq v(x,t)$ in $Q_T\cup\Gamma_T.$
\end{theorem}
The proof of the following statement relies on the continuation
principle and the construction of a supersolution.
\begin{theorem}
Let $\max(p,l)\leq1$. Then problem (\ref{v:u})--(\ref{v:n}) has
global solution for any initial datum.
\end{theorem}
\begin{proof}
In order to prove global existence of solution
for~(\ref{v:u})--(\ref{v:n}) we construct a suitable explicit
supersolution of~(\ref{v:u})--(\ref{v:n}) in $Q_T$ for any
positive $T.$ Since $c(x,t)$ and $k(x,y,t)$ are continuous
functions there exists a constant $M>0$ such that $c(x,t)\leq M$
and $k(x,y,t)\leq M$ in $Q_T$ and $\partial\Omega\times Q_T,$
respectively. Let $\lambda_1$ be the first eigenvalue of the
following problem
\begin{equation*}
    \begin{cases}
        \Delta\varphi+\lambda\varphi=0,\;x\in\Omega,\\
        \varphi(x)=0,\;x\in\partial\Omega,
    \end{cases}
\end{equation*}
and $\varphi(x)$ be the corresponding eigenfunction with
$\sup\limits_{\Omega}\varphi(x)=1$. It is well known
$\varphi(x)>0$ in $\Omega$ and $\max\limits_{\partial\Omega}
\partial\varphi(x)/\partial\nu < 0.$

Now we show that $\overline u=d\exp[bt-a\varphi(x)]$ is a
supersolution of~(\ref{v:u})--(\ref{v:n}) in $Q_T$, where
constants $a,\;b$ and $d$ are satisfied the following
inequalities:
\begin{equation*}
      b\geq a^2\sup\limits_{\Omega}|\nabla\varphi|^2+a\lambda_1+M\max\{1,d^{p-1}\exp[a(1-p)]\},
\end{equation*}
\begin{equation*}
    a\geq M|\Omega|d^{l-1}\max\limits_{\partial\Omega}\left(-\frac{\partial\varphi}{\partial\nu}\right)^{-1},\;d\geq\exp[a]\sup\limits_\Omega u_0(x).
\end{equation*}
Here $|\Omega|$ is the Lebesgue measure of $\Omega$. It is easy to
check that
\begin{equation*}
    \overline u_t-\Delta\overline u-c(x,t)\overline u^p=\left(b-a^2|\nabla\varphi|^2
    -a\lambda_1\varphi-c(x,t)\overline u^{p-1}\right)\overline u\geq0\textrm{ in }Q_T,
\end{equation*}
\begin{eqnarray*}
&&\frac{\partial \overline u}{\partial\nu} - \int_{\Omega}
k(x,y,t)\overline u^l(y,t) \,dy =
d \exp[bt]\left( a\exp[-a\varphi(x)]\left(-\frac{\partial\varphi}{\partial\nu} \right) \right. \\
&& \left. -d^{l-1}\exp[b(l-1)t]\int_{\Omega}{k(x,y,t)\exp[-al\varphi(y)]}\,dy\right) \geq 0\textrm{ on }S_T \\
\end{eqnarray*}
and
\begin{equation*}
    \overline u(x,0)\geq u_0(x)\textrm{ in }\Omega. \Box
\end{equation*}
\end{proof}
\section{Blow-up and global existence for $\max(p,l)>1$}\label{mm}

We set
$$w(t)=\int_\Omega{u(x,t)}\,dx, \,
c_0(t)=|\Omega|^{1-p}\inf\limits_{\Omega}c(x,t), \, k_0(t) =
|\Omega|^{1-l} \inf\limits_{\Omega} \int_{\partial\Omega} k(x,y,t)
\,dS_x
$$
and consider the Cauchy problem given by one of the following
equations:
\begin{equation}\label{gl:eq1:odu}
w'(t)=c_0(t)w^p,\;p>1,
\end{equation}
\begin{equation}\label{gl:eq2:odu}
w'(t)=k_0(t)w^l,\;l>1,
\end{equation}
\begin{equation}\label{gl:eq3:odu}
w'(t)=c_0(t)w^p+k_0(t)w^l,\;p\geq1,\;l\geq1,
\end{equation}
with initial datum
\begin{equation}\label{gl:id:odu}
    w(0)=\int_\Omega{u_0(x)}\,dx.
\end{equation}
\begin{theorem}
Let $\max(p,l)>1$ and the Cauchy problem~(\ref{gl:eq1:odu}),
(\ref{gl:id:odu}) ((\ref{gl:eq2:odu}), (\ref{gl:id:odu})
or~(\ref{gl:eq3:odu}), (\ref{gl:id:odu})) does not have a global
solution. Then the solution of~(\ref{v:u})--(\ref{v:n}) blows up
in finite time.
\end{theorem}
\begin{proof}
Suppose that the Cauchy problem~(\ref{gl:eq1:odu}),
(\ref{gl:id:odu}) does not have a global solution. Integrating (\ref{v:u}) over $\Omega$ and using Green's identity, we have
   \begin{eqnarray}\label{gl:eq}
   \hskip-15pt
w'(t) &=& \int_\Omega{c(x,t)u^p(x,t)}\,dx + \int_{\partial\Omega}
\int_\Omega{k(x,y,t)u^l(y,t)}\,dy\,dS_x \nonumber\\
&\geq&
\inf\limits_{\Omega}c(x,t)\int_\Omega
u^p(x,t)\,dx+\inf\limits_{\Omega} \int_{\partial\Omega}
k(x,y,t)\,dS_x \int_\Omega u^l(y,t)\,dy.
    \end{eqnarray}
By Jensen's inequality $w'(t)\geq c_0(t)w^p$. Now the comparison
principle for ordinary differential equation implies the claim.
The proof for other cases is similar. \qed
\end{proof}
\begin{remark}\label{gl:blow-up}
Solving the Cauchy problem for equations
(\ref{gl:eq1:odu})--(\ref{gl:eq3:odu}) with initial datum
(\ref{gl:id:odu}), we obtain that (\ref{v:u})--(\ref{v:n}) does
not have global solutions in the following cases:
\begin{equation*}
p>1 \text{ and }
\int_\Omega{u_0(x)}\,dx>\left((p-1)\int_0^\infty
c_0(t)\,dt\right)^{-1/(p-1)},
\end{equation*}
\begin{equation*}
l>1 \text{ and }
\int_\Omega{u_0(x)}\,dx>\left((l-1)\int_0^\infty
k_0(t)\,dt\right)^{-1/(l-1)},
\end{equation*}
\begin{equation*}
p=1, l>1  \text{ and }    \int_\Omega{u_0(x)}\,dx>\left((l-1)
\int_0^\infty k_0(t)\exp[(l-1)\int_0^t{c_0(s)\,ds}]\,dt\right)^{-1/(l-1)},
\end{equation*}
\begin{equation*}
l=1, p>1 \text{ and }
\int_\Omega{u_0(x)}\,dx>\left((p-1)\int_0^\infty
c_0(t)\exp[(p-1)\int_0^t{k_0(s)\,ds}]\,dt\right)^{-1/(p-1)},
\end{equation*}
$p>1, l>1$ and
\begin{equation*}
\int_\Omega{u_0(x)}\,dx > \min \left\{ \left((p-1)\int_0^\infty
c_0(t)\,dt\right)^{-1/(p-1)}, \left((l-1)\int_0^\infty
k_0(t)\,dt\right)^{-1/(l-1)} \right\}.
\end{equation*}
In particular, there are not nontrivial global solutions
of~(\ref{v:u})--(\ref{v:n}) if $p>1$ and
\begin{equation}\label{gl:c0-blow}
    \int_0^\infty c_0(t)\,dt=\infty,
\end{equation}
$l>1$ and
\begin{equation}\label{gl:k01-blow}
    \int_0^\infty k_0(t)\,dt=\infty,
\end{equation}
$p=1,\;l>1$ and
\begin{equation}\label{gl:k02-blow}
    \int_0^\infty k_0(t)\exp[(l-1)\int_0^t{c_0(s)\,ds}]\,dt=\infty,
\end{equation}
$p>1,\;l=1$ and
\begin{equation}\label{gl:c02-blow}
    \int_0^\infty c_0(t)\exp[(p-1)\int_0^t{k_0(s)\,ds}]\,dt=\infty,
\end{equation}
$p>1,\; l>1$ and
\begin{equation*}
    \int_0^\infty ( c_0(t) +  k_0(t) ) \, dt=\infty.
\end{equation*}
\end{remark}

Now we obtain sufficient conditions for blow-up of all nontrivial
solutions of~(\ref{v:u})--(\ref{v:n}) for $p>1$
if~(\ref{gl:c0-blow}) is not fulfilled as well as for $l>1$
if~(\ref{gl:k01-blow}) is not fulfilled. Let us introduce the
following auxiliary functions
$$
\overline{c}(t) = \int_{\Omega} c(x,t) \, dx, \quad
\overline{k}(t) = \int_{\partial\Omega} \int_{\Omega} k(x,y,t) \,
dy\,dS_x.
$$
\begin{theorem}\label{gl:theorem:l<1_p<1}
Problem~(\ref{v:u})--(\ref{v:n}) does not have nontrivial global solutions\\
if $p>1$ and
\begin{equation}\label{gl:p>1_l<1}
\int_0^\infty c_0(t)\,dt<\infty \textrm{ and
}\lim\limits_{t\to\infty}{\int_0^t{\overline{k}(\tau)}\,d\tau\left(\int_t^\infty{c_0(\tau)}\,d\tau\right)^{1/{(p-1)}}}=\infty,
\end{equation}
or $l>1$ and
\begin{equation*}
\int_0^\infty k_0(t)\,dt<\infty\textrm{ and
}\lim\limits_{t\to\infty}{\int_0^t{\overline{c}(\tau)}\,d\tau\left(\int_t^\infty{k_0(\tau)}\,d\tau\right)^{1/{(l-1)}}}=\infty.
\end{equation*}
\end{theorem}
\begin{proof}
Suppose $p>1$ and (\ref{gl:p>1_l<1}) holds. Let $u(x,t)$ be
nontrivial solution of~(\ref{v:u})--(\ref{v:n}) in $Q_T.$ Then by
Theorem~\ref{p:theorem:positive} we conclude
$u(x,\varepsilon)\geq\alpha$ for any $x\in\Omega$ and some
$\varepsilon>0$ and $\alpha >0$.  It is easy to see that
$\underline{v}(x,t)=\alpha$ is a subsolution of the following
problem
   \begin{eqnarray*}
    \left\{ \begin{array}{ll}
    v_{t} =\Delta v+c(x,t)v^p,\;x \in\Omega,\;\varepsilon<t<T,\\
    \frac{\partial v(x,t)}{\partial\nu}=\int_{\Omega}{k(x,y,t)v^l(y,t)}\,dy,\;x\in\partial\Omega,\;\varepsilon<t<T,\\
    v(x,\varepsilon)=u(x,\varepsilon),\;x\in\Omega.
    \end{array} \right.
    \end{eqnarray*}
Then by comparison principle $u(x,t)\geq \alpha$ for $t\in[\varepsilon,T)$.
Using (\ref{gl:eq}) and Jensen's inequality, we get
\begin{equation}\label{gl:444}
w'(t)\geq c_0(t)w^p+\alpha^l \overline{k}(t)\textrm{ for
}\varepsilon\leq t<T.
\end{equation}
From~(\ref{gl:444}) we have
\begin{equation}\label{gl:1}
    w'(t)\geq c_0(t)w^p
\end{equation}
and
\begin{equation}\label{gl:2}
    w'(t)\geq \alpha^l \overline{k}(t).
\end{equation}
By~(\ref{gl:p>1_l<1}) we can chose a constant $t_0>0$ such that
\begin{equation}\label{gl:333}
\left(\alpha^l
(p-1)^{1/(p-1)}\right)^{-1}<\int_\varepsilon^{t_0}{\overline{k}(\tau)}\,d\tau\left(\int_{t_0}^\infty{c_0(\tau)}\,d\tau\right)^{1/{(p-1)}}.
\end{equation}
Obviously, we need consider the case $t_0<T$.
For $\varepsilon\leq t_0<t<T$ from~(\ref{gl:1}), (\ref{gl:2}) we conclude
\begin{equation}\label{gl:3}
    w(t)\geq\left(w^{-(p-1)}(t_0)-(p-1)\int_{t_0}^t{c_0(\tau)}\,d\tau\right)^{-1/(p-1)}
\end{equation}
and
\begin{equation}\label{gl:4}
    w(t_0) \geq \alpha^l \int_{\varepsilon}^{t_0}{\overline{k}(\tau)}\,d\tau.
\end{equation}
By virtue of~(\ref{gl:333}), (\ref{gl:4}) we have
\begin{equation}\label{gl:5}
w^{-(p-1)}(t_0) < (p-1)\int_{t_0}^\infty{c_0(\tau)}\,d\tau.
\end{equation}
From~(\ref{gl:3}), (\ref{gl:5}) we can see that a solution of~(\ref{v:u})--(\ref{v:n}) blows up in finite time.
The proof of the second part of
the theorem is similar. \qed
\end{proof}

To obtain sufficient conditions for the existence of bounded
solutions for~(\ref{v:u})--(\ref{v:n}) we consider the following
auxiliary problem
\begin{equation}\label{gl:vsp-problem}
\left\{
  \begin{array}{ll}
    v_t=\Delta v,\;x\in\Omega,\;t>0 \\
    \frac{\partial v(x,t)}{\partial \nu}=g(t),\;x\in\partial\Omega,\;t>0, \\
    v(x,0)=v_0(x),\; x\in\Omega.
  \end{array}
\right.
\end{equation}
With respect to the data of~(\ref{gl:vsp-problem}) we suppose:
\begin{equation}\label{gl:1g(t)}
    g(t)\in C([0,+\infty)),\;g(t)\geq0,
\end{equation}
\begin{equation}\label{gl:v0}
    v_0(x)\in C^1(\overline\Omega),\;v_0(x)\geq0 \textrm{ in }
\Omega, \;\frac{\partial v_0(x)}{\partial\nu}=g(0)  \textrm{ on }
\partial\Omega.
\end{equation}
\begin{lemma}\label{gl:vsp_lemma}
Let~(\ref{gl:1g(t)}), (\ref{gl:v0}) hold. Then a solution
of~(\ref{gl:vsp-problem}) is bounded if and only if
\begin{equation}\label{gl:1g}
\int_0^\infty{g(t)}\,dt<\infty
\end{equation}
and there exist positive constants $\alpha,\;t_0$ and $c$ such that $\alpha>t_0$ and
\begin{equation}\label{gl:2g}
    \int_{t-t_0}^t {\frac{g(\tau)}{\sqrt{t-\tau}}}\,d\tau\leq c\textrm{ for any }t\geq\alpha.
\end{equation}
\end{lemma}
\begin{proof}
Let $G_N(x,y;t-\tau)$ be the Green function of the heat equation
with homogeneous Neumann boundary condition. We note that
$G_N(x,y;t-\tau)$ has the following properties (see, for example,
\cite{Kahane}):
    \begin{equation}\label{gl:1G_N}
        G_N(x,y;t-\tau)\geq0,\;x,y\in\Omega,\;0\leq\tau<t<T,
    \end{equation}
    \begin{equation}\label{gl:2G_N}
    \int_{\Omega}{G_N(x,y;t-\tau)}\,dy=1,\;x\in\Omega,\;0\leq\tau<t<T.
    \end{equation}
Moreover, similarly as in~\cite{Hu_Yin1} and \cite{Kahane} we can show
 \begin{equation}\label{gl:22G_N}
        |G_N(x,y;t-\tau)-1/|\Omega||\leq c_1\exp[-c_2(t-\tau)],\;x,\,y\in\Omega,\; t-\tau \geq
        \varepsilon,
    \end{equation}
\begin{equation}\label{gl:3G_N}
\frac{c_3}{\sqrt{t-\tau}}\leq\int_{\partial\Omega}G_N(x,y;t-\tau)\,dS_y
\leq \frac{c_4}{\sqrt{t-\tau}},\;x\in\partial\Omega,\;0< t-\tau
\leq \varepsilon,
\end{equation} for some small $\varepsilon>0$ and
    \begin{equation}\label{gl:4G_N}
        \int_{\partial\Omega}{G_N(x,\xi;t-\tau)}\,dS_\xi\geq c_5,\;x\in\overline\Omega,\;0\leq\tau<t<T.
    \end{equation}
Here and subsequently by $c_i\,(i\in \mathbb N)$ we denote
positive constants. Note that the upper bound in~(\ref{gl:3G_N})
is true for any $x\in\overline\Omega$. It is well known that
problem~(\ref{gl:vsp-problem}) is equivalent to the equation
    \begin{equation}\label{gl:bound}
        v(x,t)=\int_\Omega{G_N(x,y;t)v_0(y)}\,dy+\int_0^t{g(\tau)\int_{\partial\Omega}G_N(x,y;t-\tau)}\,dS_y\,d\tau,\;x\in\overline\Omega,\;t>0.
    \end{equation}
Using~(\ref{gl:1g}) --(\ref{gl:3G_N}), (\ref{gl:bound}) for some
$\varepsilon>0$ we get
    \begin{eqnarray*}
        v(x,t)&\leq& \sup_\Omega v_0(x)+|\partial\Omega|\int_0^{t-\varepsilon}\left(c_1\exp[-c_2(t-\tau)]+1/|\Omega|\right)g(\tau)\,d\tau\nonumber\\
        &&+ c_4\int_{t-\varepsilon}^t
        {\frac{g(\tau)}{\sqrt{t-\tau}}}\,d\tau\leq
        c_6,\;x\in\overline\Omega,\;t\geq\alpha.
    \end{eqnarray*}
Hence, $v(x,t)$ is a bounded solution of (\ref{gl:vsp-problem}). Necessity of~(\ref{gl:1g}), (\ref{gl:2g}) for
boundedness of a solution of~(\ref{gl:vsp-problem}) is proved similarly.
\end{proof}
\begin{remark}\label{4:zam}
The function $g(t)$ satisfies~(\ref{gl:2g}) if there exist
positive constants $\alpha,\;t_0$ and $c$ such that $\alpha>t_0$
and for some $q>2$ the inequality
        \begin{equation}\label{gl:g^p}
        \int\limits_{t-t_0}^tg^q(\tau)\,d\tau\leq c
        \end{equation}
holds  for any $t\geq\alpha.$ Indeed, applying H$\ddot{o}$lder's
inequality, we obtain
     \begin{equation*}
\int_{t-t_0}^t
{\frac{g(\tau)}{\sqrt{t-\tau}}}\,d\tau\leq\left(\int_{t-t_0}^t{g^q(\tau)}\,d\tau\right)^{1/q}\left(\int_{t-t_0}^t{(t-\tau)^{-m/2}}\,d\tau\right)^{1/m}\leq
c(t_0),
    \end{equation*}
where $1/q+1/m=1$, and hence $1<m<2$.
Now we construct a function which demonstrates
that~(\ref{gl:1g(t)}), (\ref{gl:1g}) and (\ref{gl:g^p}) with $q=2$
do not guarantee~(\ref{gl:2g}). Denote
$O_n=[n-1/n^3,n],\;n=2,3,\dots$, and consider the following
function
    \begin{equation*}
    g(t)=
    \left\{
        \begin{array}{ll}
        g_n(t),\;t\in O_n,\;n=2,3,\dots,\\
        f(t),\;t\in[0,+\infty)\backslash\bigcup_{n=2}^\infty O_n,
        \end{array}
    \right.
    \end{equation*}
where
$$g_n(t)=\frac{1}{\sqrt{n+1/n^6-t}|\ln(n+1/n^6-t)|^\alpha},\;\alpha\in(1/2,1),\;n=2,3,\dots,$$
$f(t)$ is a continuous function such that $g(t)$
satisfies~(\ref{gl:1g(t)}) and
    \begin{equation*}
\int\limits_{[0,+\infty)\backslash\bigcup_{n=2}^\infty O_n}\left(
f^2(t)+f(t) \right)\,dt<\infty.
    \end{equation*}
A straightforward computations show that $g(t)$
satisfies~(\ref{gl:1g}), (\ref{gl:g^p}) with $q=2$ and
    \begin{equation*}
\int_{n-1/n^3}^{n}
{\frac{g(\tau)}{\sqrt{n-\tau}}}\,d\tau\to\infty\textrm{ as
}n\to\infty.
    \end{equation*}
\end{remark}
Put $c_1(t)=\sup\limits_{\Omega}c(x,t)$ and
$k_1(t)=\sup\limits_{\partial\Omega\times\Omega}k(x,y,t)$. Suppose
that $c_1(t)$ and $k_1(t)$ satisfy the following conditions:
\begin{equation}\label{gl:c(x,t)}
   \int_0^\infty \left(c_1(t)+k_1(t)\right)\,dt<\infty
\end{equation}
and there exist positive constants $\alpha,\;t_0$ and $K$ such
that $\alpha>t_0$ and
\begin{equation}\label{gl:k(x,y,t)}
\int\limits_{t-t_0}^t\frac{k_1(\tau)}{\sqrt{t-\tau}}\,d\tau\leq
K\textrm{ for any }t\geq\alpha.
\end{equation}\label{gl:theorem:1}
\begin{theorem}\label{gl:theorem:l>1_p>1}
     Let $\min(p,l)>1$ and~(\ref{gl:c(x,t)}), (\ref{gl:k(x,y,t)}) hold. Then problem~(\ref{v:u})--(\ref{v:n}) has bounded global solutions for small initial data.
\end{theorem}
\begin{proof}
Let $v(x,t)$ be a solution of~(\ref{gl:vsp-problem}) with boundary
condition $g(t)=k_1(t)$ and positive initial datum. According to
Lemma~\ref{gl:vsp_lemma} there exists positive constant $V$ such
that $v(x,t)\leq V$ for any $x\in\Omega$ and $t\geq0$.

To prove the theorem we construct a supersolution
of~(\ref{v:u})--(\ref{v:n}) in the following form $\overline
u(x,t)=af(t)v(x,t),$ where
    \begin{equation*}
       f(t)=\left(A-(p-1)a^{p-1}V^{p-1}\int_0^t c_1(t)\,dt\right)^{-1/(p-1)},
    \end{equation*}
$A=1+(p-1)a^{p-1}V^{p-1}\int_0^\infty c_1(t)\,dt$ and $a$ is some positive constant.

After simple computations it follows that
    \begin{eqnarray*}
        \overline u_t-\Delta \overline u-c(x,t)\overline u^p&=&af'(t)v+afv_t-af\Delta v-a^pc(x,t)f^pv^p\\
        &\geq&av(f'(t)-a^{p-1}V^{p-1}c_1(t)f^p)=0,\;x\in\Omega,\;t>0,
    \end{eqnarray*}
and
    \begin{equation*}
\frac{\partial\overline u}{\partial\nu}-\int_\Omega
k(x,y,t)\overline u^l(y,t)\,dy\geq
af(t)k_1(t)(1-a^{l-1}f^{l-1}V^l|\Omega|)\geq0,\;x\in\partial\Omega,\;t>0,
    \end{equation*}
for small  values of $a.$ Hence, $\overline u(x,t)$ is a
supersolution of problem~(\ref{v:u})--(\ref{v:n}) for an initial
datum $u_0(x)\leq aA^{-1/(p-1)}v(x,0).$
\end{proof}
\begin{remark}
By Remark~\ref{gl:blow-up} and Theorem~\ref{gl:theorem:l>1_p>1}
condition~(\ref{gl:c(x,t)}) is optimal for global existence of
solutions for~(\ref{v:u})--(\ref{v:n}) with $c(x,t) = c(t)$ and
$k(x,y,t) = k(t).$ Arguing in the same way as in the proof of
Lemma~\ref{gl:vsp_lemma} it is easy to show
that~(\ref{gl:k(x,y,t)}) is optimal for the existence of bounded
global solutions for~(\ref{v:u})--(\ref{v:n}) with
$k(x,y,t)=k(t).$
 \end{remark}
\begin{remark}
Assume that $\min(p,l)>1$,~(\ref{gl:c(x,t)}) holds and there exist
positive constants $\alpha,\;t_0$ and $K$ such that $\alpha>t_0$
and for some $q>2$ the inequality
        \begin{equation*}
\int\limits_{t-t_0}^tk_1^q(\tau)\,d\tau\leq K
        \end{equation*}
holds for any $ t\geq\alpha.$ Then by Remark~\ref{4:zam} and
Theorem~\ref{gl:theorem:l>1_p>1} problem~(\ref{v:u})--(\ref{v:n})
has bounded global solutions for small initial data.
\end{remark}

\subsection{The case $p=1$ and $l>1$}\label{4}

Suppose that for some $K \geq 0$ and $\varepsilon>0$ the functions
$k(x,y,t)$ and $c_1(t)$ satisfy
\begin{equation}\label{gl:k(x,y,t)2}
    \int_\Omega k(x,y,t)\,dy\leq K\exp \left[-(l-1) \left\{\int\limits_0^tc_1(\tau)\,d\tau+\varepsilon t\right\}\right]
\end{equation}
for any $x\in\partial\Omega$ and $t \geq 0.$
\begin{theorem}\label{4:th}
     Let $p=1,\;l>1$ and~(\ref{gl:k(x,y,t)2}) hold. Then problem~(\ref{v:u})--(\ref{v:n}) has global solutions for small initial data.
\end{theorem}
\begin{proof}
    Let $\psi(x)$ be a solution of the following problem
    \begin{equation*}
\Delta\psi=1,\;x\in\Omega,\;\frac{\partial\psi(x)}{\partial
\nu}=\gamma,\;x\in\partial\Omega,
    \end{equation*}
where $\gamma=|\Omega|/|\partial\Omega|.$ To prove the theorem we
construct a supersolution of~(\ref{v:u})--(\ref{v:n}) in such a
form that $v(x,t)=b\exp[f(t)]\psi(x)$, where
    $f(t)=\int\limits_0^tc_1(\tau)\,d\tau+\varepsilon t$
and $b$ is some positive constant. Indeed, we have
    \begin{equation*}
v_t-\Delta
v-c(x,t)v=(c_1(t)+\varepsilon)v-\frac{v}{\psi(x)}-c(x,t)v\geq
\left(\varepsilon-\frac{1}{\psi(x)}\right)v\geq0,\;x\in\Omega,\;t>0,
    \end{equation*}
for large values of $\inf\limits_\Omega\psi(x)$ and
    \begin{eqnarray*}
        \frac{\partial v}{\partial\nu}&=&\gamma b\exp[f(t)]-b^l\exp[lf(t)]\sup\limits_{\partial\Omega}\int_\Omega k(x,y,t)\psi^l(y)\,dy\\
        &\geq&b\exp[f(t)]\left(\gamma-Kb^{l-1}\sup\limits_\Omega\psi^l(x)\right)\geq0,\;x\in\partial\Omega,\;t>0,
    \end{eqnarray*}
for small values of $b.$ Consequently, $v(x,t)$ is a supersolution
of~(\ref{v:u})--(\ref{v:n}) for an initial datum $u_0(x)\leq
b\psi(x).$
\end{proof}
\begin{remark}
It is easy to see from Remark~\ref{gl:blow-up} that
Theorem~\ref{4:th} does not hold for $\varepsilon=0.$
\end{remark}

Suppose that $k_1(t)$ and $c_1(t)$ satisfy
\begin{equation}\label{gl:k(x,y,t)3}
    \int_0^\infty k_1(t)\exp \left[ (l-1)\int_0^t c_1(\tau)\,d\tau \right]\,dt<\infty
\end{equation}
and there exist positive constants $\alpha,\;t_0$ and $K$ such
that $\alpha>t_0$ and
\begin{equation}\label{gl:k(x,y,t)4}
\int\limits_{t-t_0}^t\frac{k_1(\tau)\exp \left[ (l-1)\int_0^\tau
c_1(s)\,ds \right]}{\sqrt{t-\tau}}\,d\tau\leq K\textrm{ for any
}t\geq\alpha.
\end{equation}
\begin{theorem}\label{gl:th:2}
Let $p=1,\;l>1$ and~(\ref{gl:k(x,y,t)3}), (\ref{gl:k(x,y,t)4})
hold. Then problem~(\ref{v:u})--(\ref{v:n}) has bounded global
solutions for small initial data.
\end{theorem}
\begin{proof}
Let $v(x,t)$ be a solution of~(\ref{gl:vsp-problem}) with boundary
condition $g(t)=k_1(t)\exp[(l-1)\int_0^t c_1(\tau)\,d\tau]$ and
positive initial datum.  According to Lemma~\ref{gl:vsp_lemma}
there exists positive constant $V$ such that $v(x,t)\leq V$ for
any $x\in\Omega$ and $t \geq 0.$

To prove the theorem we construct a supersolution
of~(\ref{v:u})--(\ref{v:n}) in the following form
    \begin{equation*}
\overline u(x,t)=a \exp \left[\int_0^t c_1(\tau)\,d\tau \right]
v(x,t),
    \end{equation*}
where $a$ is some positive constant. It is easy to check that
\begin{eqnarray*}
&&\overline u_t-\Delta \overline u-c(x,t)\overline u=ac_1(t)v\exp
\left[\int_0^t c_1(\tau)\,d\tau \right]+a\exp \left[\int_0^t
c_1(\tau)\,d\tau \right]v_t\\
&&-a\exp\left[\int_0^t
c_1(\tau)\,d\tau\right]\Delta v-ac(x,t)\exp\left[\int_0^t
c_1(\tau)\,d\tau\right]v\geq0,\;x\in\Omega,\;t>0,
\end{eqnarray*}
and
    \begin{equation*}
\frac{\partial\overline u}{\partial\nu}-\int_\Omega
k(x,y,t)\overline u^l(y,t)\,dy\geq ak_1(t)\exp
\left[l\int_0^tc(\tau)\,d\tau\right]\left(1-a^{l-1}V^l|\Omega|\right)
\geq 0
    \end{equation*}
for $x\in\partial\Omega,\;t>0$ and some small  values of $a.$
Hence,  $\overline u(x,t)$ is a supersolution of
problem~(\ref{v:u})--(\ref{v:n}) for an initial datum $u_0(x)\leq
av(x,0).$
\end{proof}
\begin{remark}
By Remark~\ref{gl:blow-up} and Theorem~\ref{gl:th:2}
condition~(\ref{gl:k(x,y,t)3})  is optimal for global existence of
solutions for~(\ref{v:u})--(\ref{v:n}) with $c(x,t) = c(t)$ and
$k(x,y,t) = k(t).$ Arguing in the same way as in the proof of
Lemma~\ref{gl:vsp_lemma} it is easy to show that
(\ref{gl:k(x,y,t)4}) is optimal for the existence of bounded
global solutions for~(\ref{v:u})--(\ref{v:n}) with
$k(x,y,t)=k(t).$
\end{remark}
\begin{remark}
Assume $p=1,\;l>1$,~(\ref{gl:k(x,y,t)3}) holds and there exist
positive constants $\alpha,\;t_0$ and $K$ such that $\alpha>t_0$
and for some $q>2$ the inequality
    \begin{equation*}
\int\limits_{t-t_0}^tk_1^q(\tau)\exp \left[ q(l-1)\int_0^\tau
c_1(s)\,ds\right]\,d\tau\leq K
    \end{equation*}
holds for any $t\geq\alpha.$ Then by Remark~\ref{4:zam} and
Theorem~\ref{gl:th:2} problem~(\ref{v:u})--(\ref{v:n}) has bounded
global solutions for small initial data.
\end{remark}

\subsection{The case $l=1$ and $p>1$}\label{5}

Consider the auxiliary problem
    \begin{equation}\label{4:vz}
    \left\{
        \begin{array}{ll}
\Delta h(x)=ah(x),\;x\in\Omega,\\
\frac{\partial h(x)}{\partial \nu}=g(x)\int_\Omega
h(y)\,dy,\;x\in\partial\Omega.
        \end{array}
    \right.
    \end{equation}
With respect to the data of~(\ref{4:vz}) we suppose
\begin{equation*}
g(x)\in
C(\partial\Omega),\;g(x)\geq0,\,a=\int_{\partial\Omega}g(x)\,dS>0.
\end{equation*}
\begin{lemma}
Problem~(\ref{4:vz}) has infinitely many nonnegative solutions.
\end{lemma}
\begin{proof}
Let $\int_\Omega h(y)\,dy \neq 0.$  Set $v(x)=h(x)/\int_\Omega
h(y)\,dy.$ It is easy to verify that~(\ref{4:vz}) is reduced to
the following problem
    \begin{equation}\label{4:vz2}
    \left\{
        \begin{array}{ll}
        \Delta v(x)=av(x),\;x\in\Omega,\\
        \frac{\partial v(x)}{\partial \nu}=g(x),\;x\in\partial\Omega,
        \end{array}
    \right.
    \end{equation}
provided
     \begin{equation}\label{4:sogl}
        \int_\Omega v(y)\,dy=1.
     \end{equation}
By~\cite{Miranda} problem~(\ref{4:vz2}) has unique solution.
Obviously, this solution satisfies~(\ref{4:sogl}). Now we show
nonnegativity of $v(x)$ in $\Omega.$ Indeed, by the strong maximum
principle $v(x)$ cannot attain a negative minimum in $\Omega$.
Suppose there exists a point $x_0\in\partial\Omega$ such that
$v(x_0)=\min\limits_{\overline\Omega}v(x)<0$. Then $\partial
v(x_0)/\partial\nu<0$ (see~\cite{Friedman}) which contradicts the
boundary condition. Obviously, $h(x)=\alpha v(x)$  is nonnegative
solution of~(\ref{4:vz}) for any $\alpha>0.$
\end{proof}
Suppose that $k(x,y,t)$ and  $c_1(t)$ satisfy the following conditions:
\begin{equation}\label{4:k(x,y,t)}
k(x,y,t)\leq k_2(x), \, x \in \partial \Omega, \, y \in \Omega,\,t>0,
\end{equation}
and
\begin{equation}\label{4:c(x,t)}
   \int_0^\infty c_1(t)\exp \left[(p-1)t \int_{\partial\Omega}
   k_2(x)\,dS\right]\,dt<\infty,
\end{equation}
where $k_2(x)$ is some nonnegative continuous function on $\partial \Omega.$

\begin{theorem}
Let $l=1,\;p>1$ and~(\ref{4:k(x,y,t)}), (\ref{4:c(x,t)}) hold.
Then problem~(\ref{v:u})--(\ref{v:n}) has global solutions for small initial data.
\end{theorem}
\begin{proof}
Let $h(x)$ be some nonnegative solution of~(\ref{4:vz}) with
$a=\int_{\partial\Omega} k_2(x)\,dS$ and $g(x)=k_2(x).$  To prove
the theorem we construct a supersolution
of~(\ref{v:u})--(\ref{v:n}) in such a form that $\overline
u(x,t)=f(t)h(x),$ where
    \begin{equation*}
        f(t)=\exp[at]\left(A-(p-1)\sup_\Omega h^{p-1}(x)\int_0^t c_1(\tau)\exp[(p-1)a\tau]\,d\tau\right)^{-1/(p-1)},
    \end{equation*}
$A=1+(p-1)\sup\limits_\Omega h^{p-1}(x)\int_0^\infty
c_1(t)\exp[(p-1)at]\,dt$. Indeed, we have
    \begin{eqnarray*}
\overline u_t - \Delta \overline u - c(x,t)\overline u^p&=&f'(t)h - a f h - c(x,t)f^ph^p\\
&\geq&h(f'(t) - a f - \sup_\Omega
h^{p-1}(x)c_1(t)f^p)=0,\;x\in\Omega,\;t>0,
    \end{eqnarray*}
and
    \begin{equation*}
\frac{\partial\overline u}{\partial\nu}-\int_\Omega
k(x,y,t)\overline u(y,t)\,dy = f(t)\left(k_2 (x) \int_\Omega
h(y)\,dy-\int_\Omega k(x,y,t)h(y)\,dy\right)\geq 0
    \end{equation*}
for $x\in\partial\Omega,\;t>0$. Hence, $\overline u(x,t)$ is a
supersolution of~(\ref{v:u})--(\ref{v:n}) for an initial datum
$u_0(x)\leq A^{-1/(p-1)}h(x).$
\end{proof}
\begin{remark}
It is easy to see that (\ref{4:c(x,t)}) and~(\ref{gl:c02-blow})
are optimal conditions for global existence and blow-up of
solutions for~(\ref{v:u})--(\ref{v:n}) if, for example,
$c(x,t)=c(t)$ and $k(x,y,t)=k(x).$
\end{remark}
\section{Blow-up on the boundary}\label{blow}

In this section we show that for problem~(\ref{v:u})--(\ref{v:n})
in the case $l>1$ and $p \leq 1$ blow-up  cannot occur at the
interior domain. We introduce the following notation
\begin{equation}\label{blow:J}
    J(t)=\int_0^t\int_\Omega u^l(x,\tau)\,dx\,d\tau.
\end{equation}
\begin{lemma}\label{blow:111}
Let $l>1$, $\inf\limits_{\partial\Omega\times Q_T}k(x,y,t)>0$ and the
solution $u(x,t)$ of (\ref{v:u})--(\ref{v:n}) blows up in $t=T$.
Then for $t \in [0,T)$
    \begin{equation}\label{blow:inq}
        J(t)\leq s \left( T-t \right)^{-1/(l-1)},\;s>0.
    \end{equation}
\end{lemma}
    \begin{proof}
It is well known that $u(x,t)$ is the solution
of~(\ref{v:u})--(\ref{v:n}) in $Q_T$ if and only if
    \begin{eqnarray}\label{blow:equat}
        u(x,t)&=&\int_\Omega{G_N(x,y;t)u_0(y)}\,dy+\int_0^t{\int_\Omega{G_N(x,y;t-\tau)c(y,\tau)u^p(y,\tau)}\,dy}\,d\tau\nonumber\\
        &&+\int_0^t{\int_{\partial\Omega}{G_N(x,\xi;t-\tau)\int_{\Omega}{k(\xi,y,\tau)u^l(y,\tau)}\,dy}}\,dS_\xi\,d\tau.
    \end{eqnarray}
By virtue of~(\ref{gl:1G_N}),
(\ref{gl:4G_N}), (\ref{blow:J}), (\ref{blow:equat}) and Jensen's inequality we have
\begin{eqnarray*}
    J'(t)&=&\int_\Omega u^l(x,t)\,dx\geq k^l\int_\Omega\left(\int_0^t{\int_{\partial\Omega}{G_N(x,\xi;t-\tau)\int_{\Omega}{u^l(y,\tau)}\,dy}}\,dS_\xi\,d\tau\right)^l\,dx\\
    &\geq&k^l|\Omega|^{1-l}\left(\int_\Omega\int_0^t{\int_{\partial\Omega}{G_N(x,\xi;t-\tau)\int_{\Omega}{u^l(y,\tau)}\,dy}}\,dS_\xi\,d\tau\,dx\right)^l\\
    &\geq&(c_5k)^l|\Omega|J^l(t),
\end{eqnarray*}
where $k=\inf\limits_{\partial\Omega\times Q_T}k(x,y,t).$ Thus,
\begin{equation}\label{blow:1}
    J'(t)\geq c_7J^l(t).
\end{equation}
Integrating~(\ref{blow:1}) over $(t,T),$ we
obtain~(\ref{blow:inq}).
\end{proof}
\begin{theorem}
Let $p\leq1$ and the conditions of Lemma~\ref{blow:111} hold. Then
blow-up can occur only on the boundary.
\end{theorem}
\begin{proof}
In the proof we shall use some arguments of~\cite{Deng_Zhao},
\cite{{Hu_Yin}}. Let $u(x,t)=\exp[ct]v(x,t)$, where $c=\sup\limits_{Q_T}
c(x,t).$ It is easy to check that $v(x,t)$ is a solution of
    \begin{equation*}\label{blow:problem}
    \left\{
        \begin{array}{ll}
        v_t=\Delta v+\exp[-(1-p)ct]c(x,t)v^p-cv,\;(x,t)\in Q_T,\\
        \frac{\partial v(x,t)}{\partial \nu}=\exp[(l-1)ct]\int_\Omega k(x,y,t)v^l(y,t)\,dy,\;(x,t)\in S_T,\\
        v(x,0)=u_0(x),\;x\in\Omega.
        \end{array}
    \right.
    \end{equation*}
Then $v(x,t)$ satisfies the following equation
    \begin{eqnarray}\label{blow:eq}
        v(x,t)&=&\int_\Omega{G_N(x,y;t)u_0(y)}\,dy\\
        &&+\int_0^t{\int_\Omega{G_N(x,y;t-\tau)\left(\exp[-(1-p)ct]c(y,\tau)v^p(y,\tau)-cv(y,\tau)\right)}\,dy}\,d\tau\nonumber\\
        &&+\int_0^t{\int_{\partial\Omega}{G_N(x,\xi;t-\tau)\int_{\Omega}{\exp[(l-1)c\tau]k(\xi,y,\tau)u^l(y,\tau)}\,dy}}\,dS_\xi\,d\tau\nonumber
    \end{eqnarray}
for $(x,t) \in Q_T.$ We now take an arbitrary
$\Omega'\subset\subset\Omega$ with $\partial\Omega'\in C^2$ such
that $\textrm{dist}(\partial\Omega,\Omega')=\varepsilon>0.$ It is
known (see, for example,~\cite{Hu_Yin1}) that
    \begin{equation}\label{blow:G_N}
    0\leq G_N(x,y;t-\tau)\leq
    c_\varepsilon,\;x\in\Omega',\;y\in\partial\Omega,\;0<\tau<t<T,
    \end{equation}
where $c_\varepsilon$ is a positive constant depending on
$\varepsilon.$ By~(\ref{gl:1G_N}), (\ref{gl:2G_N}),
(\ref{blow:inq}), (\ref{blow:eq}) and (\ref{blow:G_N}) we have
    \begin{eqnarray*}
        \sup_{\Omega'}v(x,t)&\leq&\sup_\Omega u_0(x)+c\int_0^t{\int_\Omega{G_N(x,y;t-\tau)}\,dy}\,d\tau\\
        &&+ c_\varepsilon\sup_{\partial\Omega\times Q_T}k(x,y,t)|\partial\Omega|\exp[(l-1)cT]\int_0^t{\int_{\Omega}{u^l(y,\tau)}\,dy}\,d\tau\\
        &\leq& c_7+c_8J(t)\leq c_9(T-t)^{-1/(l-1)}.
    \end{eqnarray*}
Hence,
    \begin{equation*}
        \sup_{\Omega'} u(x,t)\leq c_{10}(T-t)^{-1/(l-1)}.
    \end{equation*}
As it is shown in~\cite{Hu_Yin}, there exists a function $f(x)\in
C^2(\overline{\Omega'})$ satisfying
    \begin{equation}\label{blow:pr2}
        \Delta f -\frac{l}{l-1}\frac{|\nabla f|^2}{f}\geq-c_{11}\textrm{
        in }\Omega',\;  f(x)>0\textrm{ in }\Omega',\;f(x)=0\textrm{ on
        }\partial\Omega'.
    \end{equation}

We introduce the auxiliary function
    \begin{equation*}
    w(x,t)=c_{12}\exp[\mu t]\left(f(x)+c_{11}(T-t)\right)^{-1/(l-1)},
    \end{equation*}
where the positive constants $\mu$ and $c_{12}$ will be defined
below. By~(\ref{blow:pr2}) for $x\in\Omega'$ and $t\in[0,T)$ we
get
    \begin{eqnarray*}
        &&w_t-\Delta w-c(x,t)w^p=\mu w-c(x,t)w^p\\
        &&+\frac{w}{(l-1)[f(x)+c_{11}(T-t)]}\left(c_{11}+\Delta f-\frac{l|\nabla f|^2}{(l-1)[f(x)+c_{11}(T-t)]}\right) \geq 0
    \end{eqnarray*}
provided that
    \begin{equation*}
    \mu\geq c\left(\frac{[\sup_{\Omega'} f(x)+c_{11}T]^{1/(l-1)}}{c_{12}}\right)^{1-p}.
    \end{equation*}
Choosing $c_{12}$ such that $c_{12}>c^{1/(l-1)}_{11}c_{10}$ and
$w(x,0)\geq u(x,0)$ for $x\in\Omega',$ by comparison principle we
conclude
    \begin{equation*}
    u(x,t)\leq w(x,t)\textrm{ in }{\overline{\Omega'}}\times[0,T).
    \end{equation*}
Hence, $u(x,t)$ cannot blow up in $\Omega'\times[0,T]$. Since
$\Omega'$ is an arbitrary subset of $\Omega$, the proof is
completed.
\end{proof}


\begin{thebibliography}{99}

\bibitem{CY} Z. Cui, Z. Yang, Roles of weight functions to a nonlinear porous
medium equation with nonlocal source and nonlocal boundary
condition, J. Math. Anal. Appl. 342 (2008)  559--570.

\bibitem{CYZ} Z. Cui, Z. Yang, R. Zhang, Blow-up of solutions for nonlinear
parabolic equation with nonlocal source and nonlocal boundary
condition, Appl. Math. Comput. 224 (2013) 1--8.

\bibitem{Deng_Dong} K. Deng, Z. Dong, Blow-up for the equation with a general memory boundary condition, Comm. Pure Appl. Ahal. 11 (2012) 2147--2156.

\bibitem{Deng_Zhao} K. Deng, C.L. Zhao, Blow-up for a parabolic system coupled in an equation and a boundary condition, Proc. Royal Soc. Edinb. 131A (2001) 1345--1355.

\bibitem{FZ} Z.B. Fang, J. Zhang, Global existence and
blow-up of solutions for p-Laplacian evolution equation with
nonlinear memory term and nonlocal boundary condition, Boundary
Value Problems 2014 (2014)  1--17.

\bibitem{Friedman} A. Friedman, Partial differential equations of parabolic type, Prentice-Hall, 1964.

\bibitem{GG} Y. Gao, W. Gao, Existence and blow-up of solutions for a porous medium equation with nonlocal boundary condition, Appl. Anal. 90 (2011) 799--809.

\bibitem{GladkovGuedda} A. Gladkov, M. Guedda, Blow-up problem for semilinear heat equation with absorption and a nonlocal boundary condition, Nonlinear Anal. 74 (2011) 4573--4580.

\bibitem{GladkovKavitova} A. Gladkov, T. Kavitova, Initial boundary
value problem for a semilinear parabolic equation with nonlinear
nonlocal boundary conditions,  http://arxiv.org/abs/1412.5021.

\bibitem{GladkovKim} A. Gladkov, K. Ik Kim, Blow-up of solutions for semilinear heat equation with nonlinear nonlocal boundary condition, J. Math. Anal. Appl. 338  (2008) 264--273.

\bibitem{Hu_Yin} B. Hu, H.M. Yin, The profile near blowup time for solution of the heat equation with a nonlinear boundary condition, Trans. Amer. Math. Soc. 346 (1994) 117--135.

\bibitem{Hu_Yin1} B. Hu, H.M. Yin, Critical exponents for a system of heat equations coupled in a non-linear boundary condition, Math. Meth. Appl. Sci. 19 (1996) 1099--1120.

\bibitem{Kahane} C.S. Kahane, On the asymptotic behavior of solutions of parabolic equations, Czechoslovac Math. J. 33  (1983) 262--285.

\bibitem{Liu} D. Liu, C. Mu, Blowup properties for a semilinear reaction-diffusion system with nonlinear nonlocal boundary conditions, Abstr. Appl. Anal. 2010 (2010) 1--17.

\bibitem{Miranda} C. Miranda, Equazioni alle derivate pazzialli di tipo ellittico, Springer-Verlag, 1955.

\bibitem{Pao} C.V. Pao, Asimptotic behavior of solutions of reaction-diffusion equations with nonlocal boundary conditions, J. Comput. Appl. Math. 88 (1998) 225--238.

\bibitem{WMX} Y. Wang, C. Mu, Z. Xiang, Blowup of solutions to a porous medium equation with nonlocal boundary condition, Appl. Math. Comput. 192 (2007)
579--585.

\bibitem{YF}  L. Yang, C. Fan, Global existence and blow-up of solutions to a
degenerate parabolic system with nonlocal sources and nonlocal
boundaries, Monatshefte f\"ur Mathematik. 174 (2014) 493--510.

 \bibitem{YX} Z. Ye, X. Xu, Global existence and blow-up for a porous medium
system with nonlocal boundary conditions and nonlocal sources,
Nonlinear Anal. 82 (2013) 115--126.

\bibitem{Y} H.M. Yin, On a class of parabolic equations with nonlocal
boundary conditions, J. Math. Anal. Appl. 294 (2004) 712--728.

\bibitem{ZK}  S. Zheng, I. Kong, Roles of weight functions in a nonlinear
nonlocal parabolic system, Nonlinear Anal. 68 (2008) 2406--2416.

\end{thebibliography}
\end{document}